\documentclass{amsart}
\usepackage{amsfonts}
\usepackage{graphicx}
\usepackage{amscd}
\usepackage{amsmath}
\usepackage{amssymb}
\usepackage{xcolor}

\makeatletter
\@namedef{subjclassname@2010}{%
  \textup{2010} Mathematics Subject Classification}
\makeatother

\usepackage{mathrsfs}
\usepackage{amsbsy}

\newtheorem*{acknowledgement}{Acknowledgement}

\newtheorem{corollary}{\bf Corollary}

\newtheorem{lemma}{\bf Lemma}

\newtheorem{remark}{Remark}

\newtheorem{theorem}{\bf Theorem}

\renewcommand{\div}{{\rm div}}

\theoremstyle{definition}

\numberwithin{equation}{section}

\makeatletter
\@namedef{subjclassname@2020}{%
  \textup{2020} Mathematics Subject Classification}
\makeatother

\title[Steady Ricci solitons]{A comparison theorem for steady Ricci solitons}

\author{B. Leandro}
\author{J. Poveda}

\address{Universidade Federal de Goi\'as, IME,
	Caixa Postal 131, CEP 74690-900, Goi\^ania, GO, Brazil.}
	\email{bleandroneto@ufg.br}
	\email{jeferson.arley@discente.ufg.br}

\date{\today}

\begin{document}

\begin{abstract}
We prove that a steady gradient Ricci soliton is either Ricci flat with a constant potential function, or a quotient of the product steady soliton $N^{n-1}\times\mathbb{R}$, where $N^{n-1}$ is Ricci flat, or isometric to the Bryant soliton (up to scalings), provided that a couple of geometric conditions inspired by the cigar soliton hold. As an application, we prove that any complete non-compact steady Ricci soliton with positive Ricci curvature controlled by the scalar curvature $R$, curvature tensor $Rm$ satisfying $|Rm|r\to o(1)$ and $R\to\infty$, as $r\to\infty$, must be the Bryant soliton. Moreover, we prove that any complete steady soliton with positively pinched Ricci curvature must be Ricci flat.
\end{abstract}

\maketitle

\section{Introduction}\label{int}

A complete Riemannian metric $g$ on a smooth manifold $M^n$ is called a {\it gradient steady Ricci soliton} if there exists a smooth potential function $f$ on $M^n$ such that the Ricci tensor $Ric$ of the metric $g$ satisfies the equation
\begin{equation}
\label{Eq1}
Ric=Hess\,f.
\end{equation} Here, $Hess\,f$ denotes the Hessian of $f.$ Such a function $f$ is called a potential function of the gradient steady soliton. Clearly, when $f$ is a constant the gradient steady Ricci soliton is simply a Ricci flat manifold. Thus Ricci solitons are natural extensions of Einstein metrics. Gradient steady solitons play an important role in Hamilton’s Ricci flow \cite{Hamilton2} as they correspond to translating solutions and often arise as Type $II$ singularity models, thus playing a crucial role in the singularity analysis of the Ricci flow \cite{Perelman2}.

It is well-known that compact gradient steady solitons must be Ricci flat. For dimension $n=2,$ Hamilton {\cite{Hamilton3}} obtained the {\it cigar soliton}, i.e., the first example of a complete noncompact gradient steady soliton on $\mathbb{R}^2,$ where the explicit metric is given by $g=\frac{dx^2 + dy^2}{1+x^2 + y^2}$ and the potential function is $f= - \ln (1+x^2 + y^2).$ It has positive curvature and is asymptotic to a cylinder of finite circumference at infinity. The scalar curvature decays exponentially. It is important to highlight that $$R|\nabla f|=|\nabla R|$$ on the cigar soliton. This identity also is trivially satisfied, in higher dimensions, for Ricci flat solitons with a constant potential function, a quotient of the product steady soliton $N^{n-1}\times\mathbb{R}$, where $N^{n-1}$ is Ricci flat and the product of the cigar soliton and any complete Ricci flat manifold.

For dimension $n\ge 3,$ Robert Bryant {\cite{Bryant1}} proved that
there exists, up to scaling, a unique complete rotationally symmetric gradient Ricci soliton on $\mathbb{R}^n.$ Its sectional curvature is positive and the scalar curvature $R$ decays like $r^{-1}$ at infinity, and the volume of the geodesic balls $B_r(0)$ grows according to the order of $r^{(n+1)/2}$. Here, $r$ denotes the geodesic distance from the origin.

It was conjectured by Perelman {\cite{Perelman2}} that in dimension $n=3$ the Bryant soliton is the only complete noncompact ($\kappa$-noncollapsed) gradient steady soliton with positive sectional curvature. This conjecture was proved by Brendle {\cite{brendle2013}} in 2013 and was extended by himself in 2014 {\cite{brendle2014}} for arbitrary dimension $n\geq 4$ under the additional condition that the steady Ricci soliton is asymptotically cylindrical. Moreover, Deng and Zhu \cite{deng2020}  proved that any noncompact $\kappa$-noncollapsed steady Ricci soliton with nonnegative curvature operator must be rotationally symmetric if it has a linear curvature decay. Hence, for $n\geq 4,$  it is desirable to find geometrically interesting conditions under which the uniqueness would hold.

In this context, Cao and Chen {\cite{cao2011}} proved that a complete noncompact $n$-dimensional $(n \ge 3)$ locally conformally flat gradient steady Ricci soliton with positive sectional curvature is isometric to the Bryant soliton. Moreover, they
showed that a complete noncompact $n$-dimensional locally conformally flat gradient steady Ricci soliton is either flat or isometric to the Bryant soliton (see also Catino-Mantegazza {\cite{catino}}). For $n = 4,$ Chen and Wang {\cite{chenwang}} showed that any $4$-dimensional complete half-conformally flat gradient steady Ricci soliton is either Ricci flat, or locally conformally flat (hence isometric to the Bryant soliton). In {\cite{cao2014}}, Cao, Catino, Chen, Mantegazza and Mazzieri proved that any $n$-dimensional $(n\ge 4)$ complete Bach-flat
gradient steady Ricci soliton with positive Ricci curvature is isometric to the Bryant soliton. Very recently, Cao and Yu proved that any $n$-dimensional complete noncompact gradient steady Ricci soliton with vanishing $D$-tensor is either Ricci flat, or a quotient of the product steady soliton $N^{n-1}\times \mathbb{R},$ where $N^{n-1}$ is Ricci flat, or isometric to the Bryant soliton (up to scalings). Here, $D$-tensor is a $3$-tensor defined by \eqref{tensorD}, see \cite{cao2020}.

In recent years, a lot of progress has been made in understanding the curvature estimate of gradient steady Ricci solitons, {see}, e.g. \cite{chan2019,chow2011,deng2020,lopez2013}.

It follows from a result by Chen \cite{chen2009} that  every a complete gradient steady Ricci soliton has nonnegative scalar curvature, i.e., $R\geq 0.$ On the other hand, it is well know that a complete gradient steady Ricci soliton satisfies 
\begin{equation}
R+|\nabla f|^2 ={C_{0}}\nonumber, 
\end{equation} where $C_{0}$ is a positive constant. In other words, the scalar curvature of a complete gradient steady Ricci soliton is uniformly bounded. In particular, up to normalization, we may consider
\begin{equation}
\label{eqR1}
R+|\nabla f|^2 =1. 
\end{equation} This therefore implies that a (normalized) gradient steady Ricci soliton satisfies $0\leq R\leq 1.$

Before proceeding, we recall that Brendle \cite{brendle2011} (see also \cite[Proposition 5.2]{cao2014}) proved the following result.

\begin{theorem}[Brendle, \cite{brendle2011}]
\label{thm1}

Let $(M^n, g, f)$ ($n\ge 3$) be a complete $n$-dimensional gradient steady Ricci soliton. Suppose that the 
scalar curvature R of $(M^n, g)$ is positive and approaches zero at infinity. Denote by $\psi: (0,1) \to \mathbb R$ the smooth function such that the vector field 
$$X:=\nabla R +\psi (R)\nabla f=0$$ on the Bryant soliton, and define $u: (0,1) \to \mathbb R$ by
$$ u(s)=\log \psi (s) +\frac{1}{n-1}\int_{1/2}^{s} \left(\frac{n}{1-t} - \frac{n-1-(n-3)t}{(1-t)\psi(t)}\right) dt. $$
Moreover, assume that there exists an exhaustion of $M^n$ by bounded domains $\Omega_l$ 
such that
\begin{equation}\label{asym}
\lim_{l \to \infty}  \int_{\partial\Omega_l} e^{u(R)} \langle \nabla R +\psi (R) \nabla f, \nu\rangle \, = 0. 
\end{equation} Then $X=0$ and $D_{ijk}=0$. In particular, for $n=3$, $(M^3, g, f)$ is isometric to the Bryant soliton. 
\end{theorem}

This combined with Proposition 5.1 by Cao et al. \cite{cao2014} implies that a complete gradient steady Ricci soliton $(M^{n},g_{ij}, f),$ ($n\ge 4$), with positive Ricci curvature such that the scalar curvature $R$ approaches zero at infinity so that the condition~\eqref{asym} is satisfied for some exhaustion of $M^n$ by bounded domains $\Omega_l$ is isometric to the Bryant soliton.

We highlight that the proof of Theorem \ref{thm1} is partially based on the ideas outlined by Robinson \cite{robinson}  to study the uniqueness of static black holes, which depends essentially of a suitable divergence formula. In this paper, motivated by the result obtained by Brendle \cite{brendle2011} and the ideas by Robinson \cite{robinson}, we obtain a divergence formula (Lemma \ref{lema22}) for the steady gradient Ricci soliton by following  the ideas of Robinson in order to obtain a rigidity result. More precisely, we have established the following result. 

\begin{theorem}
\label{thmA} Let $\big(M^n,\,g,\,f)$ be a complete noncompact steady gradient Ricci soliton satisfying 
\begin{eqnarray}
\label{riccibound}
\sigma R|\nabla f|\leq|\nabla R|,
	\end{eqnarray}
	where $\sigma= \frac{(n+1)+\sqrt{(n-1)(7n-13)}}{3n-2}.$
	Suppose that there exists an exhaustion of $M^n$ by bounded domains $\Omega_\ell$ such that
\begin{eqnarray}\label{asympBene}
\displaystyle\lim_{\ell\rightarrow+\infty} \int_{\partial \Omega_\ell}|\nabla R + R\nabla f| = 0.
\end{eqnarray} Then $(M^n,\,g,\,f)$ is either Ricci flat with a constant potential function, or a quotient of the product steady soliton $N^{n-1}\times\mathbb{R}$, where $N^{n-1}$ is Ricci flat, or isometric to the Bryant soliton (up to scalings). 
\end{theorem}

\begin{remark}
\label{remarkajuda}
We point out that no sectional curvature bound is assumed. This is important because, in dimension $4$, the sectional curvature of shrinking and steady Ricci solitons may change sign. Moreover, equality in \eqref{riccibound} holds for the cigar soliton with $\sigma=1$, and \eqref{asympBene} is trivially satisfied. In that sense, our theorem is a comparison theorem with the geometry of the cigar soliton. 
\end{remark}

	\begin{remark}\label{curvature scalar bound}
	    In \cite[Lemma 3]{lopez2013}, the authors proved that a nonnegatively curved  steady soliton satisfies the following inequality: $$|Ric |^{2}\leq \frac{R^{2}}{2}.$$ Therefore,
	    $$|\nabla R|^{2}\leq 2R^2|\nabla f|^2.$$
	   This shows us that \eqref{riccibound} can be interpreted as a lower bound for $|\nabla R|$. In fact, $\sigma \leq \dfrac{1+\sqrt{7}}{3}$, see also \cite[Lemma 2.3]{deng2020}.
	\end{remark}

One possible conjecture concerning the curvature decay of a steady gradient Ricci soliton is that the decay rate is either linear or exponential \cite{chan2022}. In \cite[Corollary 1]{chan2019} the author proved the scalar curvature of a steady Ricci soliton with nonnegative Ricci curvature decays exponentially if an asymptotic condition holds.

 In \cite[Theorem 9.56]{CLN1}, Chow, Lu and Ni proved that the scalar curvature of complete and noncompact steady gradient Ricci soliton with positively pinched Ricci curvature has exponential decay. In that sense, condition \eqref{asympBene} can be replaced by an exponential decay for the scalar curvature, i.e.,
    \begin{eqnarray*}
    R = o(e^{-r}),\quad r\rightarrow\infty,
    \end{eqnarray*}
    where $r$ stands for the geodesic distance from a fixed point.

    In \cite{munteanu}, the authors proved that the sectional curvature $Rm$ of a steady gradient Ricci soliton (non-flat) with potential function $f$ bounded from above by a constant and such that $|Rm|r=o(1)$, at infinity, decays like
    \begin{eqnarray}\label{munteanu22}
    |Rm|\leq c(1+r)^{3(n+1)}e^{-r},
    \end{eqnarray}
where $c$ is a positive constant and $r$ stands for the geodesic distance. It is known that
the assumption over $f$ holds true when $Ric > 0$. The exponential decay rate in the theorem is sharp as seen from $M = N\times\Sigma$; where $\Sigma$ is the cigar soliton and $N$ a compact Ricci flat manifold.

Moreover, Chan and Zhu \cite{chan2022} proved that a steady gradient Ricci soliton $(M,\,g,\,f)$ with $|Rm|\to 0$, then either one of the following estimates holds outside a
compact set of $M$:
\begin{eqnarray*}
&&C^{-1}r^{-1} \leq |Rm| \leq Cr^{-1};\\
&&C^{-1}e^{-r} \leq |Rm| \leq Ce^{-r},
\end{eqnarray*}
where $C$ is a positive constant and $r$ is the distance function.

The control of the curvature is an important issue in the analysis of a Ricci soliton and there are several results proving that the scalar curvature controls the sectional curvature, see \cite{cui2020,chan2019}. Here is important to emphasize that four dimensional Ricci solitons must satisfies the following condition:
\begin{eqnarray*}
|Rm|\leq A\left(\frac{|\nabla Ric|}{|\nabla f|} + |Ric|\right),
\end{eqnarray*}
where $A$ is an universal positive constant (see also \cite{cui2020,munteanu2015}). Please, see also \cite[Proposition 2]{chan2019}.

Therefore, inspired by the above curvature properties of four dimensional steady Ricci solitons we will prove that if the Ricci curvature is controlled by the scalar curvature we get the Bryant soliton. Moreover, we will not assume that the steady Ricci soliton is $\kappa$-noncollapsed.

\begin{theorem}
\label{thmMunteanu} Let $\big(M^n,\,g,\,f)$ be a complete noncompact steady gradient Ricci soliton with positive Ricci curvature and
\begin{eqnarray*}
|Ric| \leq \frac{1}{4}\left[\frac{3}{2}\frac{|\nabla R|^2}{|\nabla f|^2}-2R^2\right].
\end{eqnarray*}
Suppose that
\begin{eqnarray*}
\lim_{r\to\infty}R=0\quad\mbox{and}\quad\lim_{r\to\infty}|Rm|r=o(1).
\end{eqnarray*}
 Then $(M^n,\,g,\,f)$ is isometric to the Bryant soliton (up to scalings). 
\end{theorem}

According to Hamilton \cite{Hamilton,Hamilton2}, a Riemannian manifold $(M^n,\,g)$ is {\it positively pinched Ricci
curvature} if there is a uniform constant $\delta > 0$ such that $$\delta Rg\leq Ric(g).$$ Deng and Zhu \cite{deng2015mathz} proved that any $(n\geq 2)$-dimensional complete noncompact steady K\"ahler-Ricci soliton $(M^n,\,g,\,f)$ with positively pinched Ricci curvature should be
Ricci flat (provided that there is a point $p$ so that $\nabla f(p)=0$). The existence of such a point $p$ is called equilibrium point condition. Under our approach, we shall show that this condition can be removed. To be precise, we have established the following results.

	\begin{corollary}\label{maintheoremkahler}
	Let $\big(M^n,\,g,\,f)$, $n\geq3$, be a complete noncompact steady gradient K\"ahler-Ricci soliton. If $(M^n,\,g)$ has positively pinched Ricci curvature 
	, then $(M^n,\,g,\,f)$ is Ricci flat.
	\end{corollary}

Thus, Corollary \ref{maintheoremkahler} answers a question proposed by Chow, Lu and Ni in case of steady K\"ahler-Ricci solitons (cf. \cite{CLN1,ni2005}), and proves Corollary 1.5 in \cite{deng2015mathz}. In fact, the above corollary still holds for any steady gradient Ricci soliton (not necessarily K\"ahler). 

Steady Ricci solitons with pinched Ricci curvature have been studied in the past years (cf. \cite{deng2021} and the references therein). In \cite{ni2005}, Ni proved that any steady Ricci soliton with pinched Ricci curvature and nonnegative sectional curvature must be flat. Then, Deng and Zhu \cite{deng2015mathz} proved that Ni’s result is true without the assumption over the sectional curvature for steady K\"ahler-Ricci solitons.

\begin{remark}
   We highlight that in the K\"ahler case, Cao found two examples of complete rotationally symmetric noncompact gradient steady K\"ahler-Ricci solitons (see \cite{cao2011} and the references therein). These examples are $U(n)$ invariant and have positive sectional curvature. 
\end{remark}

\section{Background}

Throughout this section we review some basic facts and present key results that will be useful in the proof of our main result. We start by recalling that for a Riemannian manifold $(M^{n},\,g),$ $n\geq 3,$ the Weyl tensor $W$ is defined by the following decomposition formula
\begin{eqnarray*}
\label{weyl}
R_{ijkl}&=&W_{ijkl}+\frac{1}{n-2}\big(R_{ik}g_{jl}+R_{jl}g_{ik}-R_{il}g_{jk}-R_{jk}g_{il}\big)\nonumber\\&&-\frac{R}{(n-1)(n-2)}\big(g_{jl}g_{ik}-g_{il}g_{jk}\big),
\end{eqnarray*}
where $R_{ijkl}$ stands for the Riemannian curvature operator. Moreover, the Cotton tensor $C$ is given according  to

\begin{equation*}
\label{cotton}
\displaystyle{C_{ijk}=\nabla_{i}R_{jk}-\nabla_{j}R_{ik}-\frac{1}{2(n-1)}\big(\nabla_{i}R g_{jk}-\nabla_{j}R g_{ik}).}
\end{equation*} 

 The Cotton and Weyl tensors are related by the following equation on steady gradient Ricci solitons (cf. \cite[Lemma 2.4]{cao2014}).

\begin{lemma}\label{lem20}
	Let $\big(M^n,\,g,\,f)$ be a steady gradient Ricci soliton. Then:
	\begin{eqnarray*}
		C_{ijk}&=&W_{ijks}\nabla^{s}f+D_{ijk},
	\end{eqnarray*} where the $D$-tensor is given by 
	\begin{eqnarray}\label{tensorD}
	D_{ijk} &=& \frac{1}{n-2}(R_{ik}\nabla_{j}f - R_{jk}\nabla_{i}f) \nonumber\\
	&+& \frac{1}{2(n-1)(n-2)}[g_{jk}(\nabla_{i}R+2R\nabla_{i}f)-g_{ik}(\nabla_{j}R+2R\nabla_{j}f)]
	\end{eqnarray}
\end{lemma}

It is well-known that a gradient Ricci soliton satisfies the equation (cf. \cite[Proposition 3]{brendle2011})
\begin{eqnarray}\label{RF1}
\nabla R=-2Ric(\nabla f).
\end{eqnarray} Moreover, it also satisfies
\begin{eqnarray}\label{Rf}
R + |\nabla f|^{2}=C_{0}
\end{eqnarray} at each any point on $M,$ where $C_{0}$ is a positive constant. Up to a normalization of $f,$ we may assume that $C_{0}=1.$ Therefore, a steady (normalized) gradient Ricci soliton satisfies
\begin{eqnarray}
R + |\nabla f|^{2}=1\nonumber.
\end{eqnarray} We remember that a complete steady gradient Ricci soliton has nonnegative scalar curvature (see \cite[Lemma 2.2]{cao2014}), i.e., $R\geq0.$ Therefore, we conclude that a steady (normalized) gradient Ricci soliton must satisfy $$0\leq R\leq 1.$$

As a consequence \eqref{tensorD} we have the following key lemma (see \cite[Proposition 4]{brendle2011}). 
\begin{lemma}\label{lem200}
	Let $\big(M^n,\,g,\,f)$ be a steady gradient Ricci soliton. Then we have:
	\begin{eqnarray*}
|D|^{2} + \frac{|\nabla R + 2R\nabla f|^{2}}{2(n-1)(n-2)^{2}}
&=& -\frac{(1-R)}{(n-2)^{2}}\Delta R-\frac{(1-R)}{(n-2)^{2}}\langle\nabla R,\,\nabla f\rangle -\frac{1}{2(n-2)^{2}}|\nabla R|^{2}.
\end{eqnarray*}
\end{lemma}
\begin{proof}
Let us rewrite the norm of $D$ only depending of the function $f$ and the scalar curvature $R$. We begin the computation using \eqref{tensorD}. To that end, we start with the following equality:
\begin{eqnarray*}
|D|^{2} &=& \frac{2}{(n-2)^{2}}|Ric|^{2}|\nabla f|^{2} -\frac{2}{(n-2)^{2}}R_{ik}\nabla_jfR_{jk}\nabla_if \nonumber\\
&+&\frac{1}{2(n-1)(n-2)^{2}}|\nabla R + 2R\nabla f|^{2}\nonumber\\
&+&\frac{2}{(n-1)(n-2)^{2}}(R_{ik}\nabla_{j}f - R_{jk}\nabla_{i}f)(\nabla_{i}R+2R\nabla_{i}f)g_{jk}.
\end{eqnarray*}
Then, by using \eqref{Eq1}, \eqref{RF1} and \eqref{Rf} we get
\begin{eqnarray*}
|D|^{2} &=& \frac{2}{(n-2)^{2}}|Ric|^{2}|\nabla f|^{2} -\frac{2}{(n-2)^{2}}\nabla_{i}\nabla_{k}f\nabla_jf\nabla_{j}\nabla_{k}f\nabla_if \nonumber\\
&+&\frac{1}{2(n-1)(n-2)^{2}}|\nabla R + 2R\nabla f|^{2}\nonumber\\
&-&\frac{1}{(n-1)(n-2)^{2}}(\nabla_{i}R + 2R\nabla_{i}f)(\nabla_{i}R+2R\nabla_{i}f)\nonumber\\
&=& \frac{2}{(n-2)^{2}}|Ric|^{2}|\nabla f|^{2} -\frac{1}{2(n-2)^{2}}|\nabla R|^{2}\nonumber\\
&-&\frac{1}{2(n-1)(n-2)^{2}}(|\nabla R|^{2}+4R\langle\nabla R,\,\nabla f\rangle+4R^{2}|\nabla f|^{2}).\nonumber\\
\end{eqnarray*}

Now, we need to use the following identity which comes from \eqref{Eq1}:
\begin{eqnarray*}\label{eq1}
2Ric(\nabla f)=\nabla |\nabla f|^{2}.
\end{eqnarray*}

 Then, taking the divergence of the above equation and again using \eqref{Eq1} and the contracted second Bianchi identity, we get
 \begin{eqnarray*}
 2|Ric|^{2}+ \langle\nabla R,\,\nabla f\rangle=\Delta|\nabla f|^{2}.
 \end{eqnarray*}
 Thus, from \eqref{Rf} we have
  \begin{eqnarray*}\label{topd}
 2|Ric|^{2}=-\langle\nabla R,\,\nabla f\rangle-\Delta R.
 \end{eqnarray*}
 Therefore, combining this identities we obtain 
 \begin{eqnarray*}
|D|^{2} + \frac{|\nabla R + 2R\nabla f|^{2}}{2(n-1)(n-2)^{2}}
&=& \frac{-1}{(n-2)^{2}}|\nabla f|^{2}\Delta R-\frac{|\nabla f|^{2}}{(n-2)^{2}}\langle\nabla R,\,\nabla f\rangle -\frac{1}{2(n-2)^{2}}|\nabla R|^{2}.\nonumber\\
\end{eqnarray*}
By using one more time \eqref{Rf} in the above equation the result follows.
\end{proof}

\begin{remark}
	We need to point out that Proposition 4 in \cite{brendle2011} follows from the above lemma considering $n=3.$ In fact, we can rewrite Lemma \ref{lem200} in the following form
	\begin{eqnarray*}
		|D|^{2} 
		&=& \frac{-(1-R)\Delta R}{(n-2)^{2}}-\frac{\langle\nabla R,\,\nabla f\rangle}{(n-2)^{2}} -\frac{n|\nabla R|^{2}}{2(n-1)(n-2)^{2}}\nonumber\\
		&+&\frac{(n-3)R\langle\nabla R,\,\nabla f\rangle}{(n-1)(n-2)^{2}}-\frac{2R^{2}(1-R)}{(n-1)(n-2)^{2}}.
	\end{eqnarray*}
	The above equation was the one used in \cite{brendle2011} for $n=3$.
\end{remark}

In what follows, we provide a new divergente formula for the steady Ricci solitons.

\begin{lemma}\label{lema22}
Let $\big(M^n,\,g,\,f)$ be a steady gradient Ricci soliton. Then,
	\begin{eqnarray*}
(1-R)^{3/2}div\left(\frac{\nabla R}{\sqrt{1-R}}-2\sqrt{1-R}\nabla f\right)=-(n-2)^{2}|D|^{2} - \frac{|\nabla R + 2R\nabla f|^{2}}{2(n-1)}-2R(1-R)^{2}.
\end{eqnarray*}
\end{lemma}
\begin{proof}
Consider $Y=\phi(R)\nabla R+ \psi(R)\nabla f$. Since from \eqref{Eq1} we have $\Delta f= R$, we get
\begin{eqnarray*}
div(Y) = \phi\Delta R + \phi'|\nabla R|^{2}+\psi'\langle\nabla R,\,\nabla f\rangle + \psi\Delta f = \phi\Delta R + \phi'|\nabla R|^{2}+ \psi'\langle\nabla R,\,\nabla f\rangle + R\psi.
\end{eqnarray*}
Thus,
\begin{eqnarray*}
2(1-R)div(Y) =  2(1-R)\phi\Delta R+2(1-R)\phi'|\nabla R|^{2} + 2(1-R)\psi'\langle\nabla R,\,\nabla f\rangle + 2(1-R)R\psi.
\end{eqnarray*}
Combining this with the previous lemma we get
	\begin{eqnarray*}
&&2\phi(n-2)^{2}|D|^{2} + \frac{\phi|\nabla R + 2R\nabla f|^{2}}{(n-1)}
= -2\phi(1-R)\Delta R-2\phi(1-R)\langle\nabla R,\,\nabla f\rangle -\phi|\nabla R|^{2}\nonumber\\
&=&-2(1-R)div(Y)+2(1-R)\phi'|\nabla R|^{2} + 2(1-R)\psi'\langle\nabla R,\,\nabla f\rangle \nonumber\\
&+& 2(1-R)R\psi-2(1-R)\phi\langle\nabla R,\,\nabla f\rangle -\phi|\nabla R|^{2}.
\end{eqnarray*}
Hence,
	\begin{eqnarray*}
	&&2\phi(n-2)^{2}|D|^{2} + \frac{\phi|\nabla R + 2R\nabla f|^{2}}{(n-1)}=-2(1-R)div(Y)+ 2(1-R)R\psi\nonumber\\
	&+&[2(1-R)\phi'-\phi]|\nabla R|^{2} + 2(1-R)[\psi'-\phi]\langle\nabla R,\,\nabla f\rangle.
\end{eqnarray*}

Consider, 
\begin{eqnarray*}
\phi= \frac{k}{\sqrt{1-R}}\quad\mbox{and}\quad \psi= -2k\sqrt{1-R},
\end{eqnarray*}
where $k$ is a nonnull constant. Therefore, 
	\begin{eqnarray*}
	2\phi(n-2)^{2}|D|^{2} + \frac{\phi|\nabla R + 2R\nabla f|^{2}}{(n-1)}+4k(1-R)^{3/2}R=-2(1-R)div(Y).
\end{eqnarray*}
\end{proof}

\section{Proof of the main result}
This section is reserved for the proofs of the main results of this paper. Let us start with the following theorem:

\begin{proof}[{\bf Proof of Theorem {\ref{thmA}}}]
From Lemma \ref{lema22} we can infer that
	\begin{eqnarray*}
(1-R)^{3/2}div\left(Y\right)=-(n-2)^{2}|D|^{2}- \frac{|\nabla R + 2R\nabla f|^{2}}{2(n-1)}-2R(1-R)^{2},
\end{eqnarray*}
where $Y=\frac{\nabla R}{\sqrt{1-R}}-2\sqrt{1-R}\nabla f.$

Let $\Omega$ be a bounded domain on $M$ with smooth boundary. Using the divergence theorem, we get
\begin{eqnarray}\label{imptinf}
 &&\int_{\partial \Omega}\langle(1-R)\nabla R-2(1-R)^{2}\nabla f,\,\nu\rangle = \int_{\Omega\cap\{R<1\}}div\left((1-R)^{3/2}Y\right) \nonumber\\
 &=&\int_{\Omega\cap\{R<1\}}(1-R)^{3/2}div\left(Y\right) - \frac{3}{2}\int_{\Omega\cap\{R<1\}}\sqrt{1-R}\langle Y,\,\nabla R\rangle\nonumber\\
  &=&-\int_{\Omega\cap\{R<1\}}\left[(n-2)^{2}|D|^{2} + \frac{|\nabla R + 2R\nabla f|^{2}}{2(n-1)}+2R(1-R)^{2}\right]\nonumber\\
  &-& \frac{3}{2}\int_{\Omega\cap\{R<1\}}\sqrt{1-R}\langle Y,\,\nabla R\rangle\nonumber\\
   &=&-\int_{\Omega\cap\{R<1\}}\left[(n-2)^{2}|D|^{2} + \frac{|\nabla R + 2R\nabla f|^{2}}{2(n-1)}+2R(1-R)^{2}\right] \nonumber\\
   &-& \frac{3}{2}\int_{\Omega\cap\{R<1\}}\langle \nabla R- 2(1-R)\nabla f,\,\nabla R\rangle\nonumber\\
 &=&-\int_{\Omega\cap\{R<1\}}\left[(n-2)^{2}|D|^{2} + \frac{|\nabla R + 2R\nabla f|^{2}}{2(n-1)}+2R(1-R)^{2}\right] \nonumber\\
 &-& \frac{3}{2}\int_{\Omega\cap\{R<1\}} [|\nabla R|^{2}- 2(1-R)\langle\nabla f,\,\nabla R\rangle]\nonumber\\
  &=&-\int_{\Omega\cap\{R<1\}}\left[(n-2)^{2}|D|^{2} + \frac{|\nabla R + 2R\nabla f|^{2}}{2(n-1)}\right] \nonumber\\
 &-& \int_{\Omega\cap\{R<1\}} [\frac{3}{2}|\nabla R|^{2}+ 6(1-R)Ric(\nabla f,\,\nabla f)+2R(1-R)^{2}],
\end{eqnarray}
where $\nu$ is the normal vector field for $\Omega$. We used in the above equation the identity
\begin{eqnarray}\label{relacaof e R}
-2Ric(\nabla f)=\nabla R.
\end{eqnarray}
Furthermore, using $R=\Delta f$ we get
\begin{eqnarray*}\label{divparalelo}
 div((1-R)^{2}\nabla f)-4(1-R)Ric(\nabla f,\,\nabla f)= R(1-R)^{2}.
\end{eqnarray*}

Thus, from the above equations we have
\begin{eqnarray}\label{intgrande}
 \int_{\partial \Omega}\langle(1-R)\nabla R-2(1-R)^{2}\nabla f,\,\nu\rangle  
 &=&-\int_{\Omega\cap\{R<1\}}\left[(n-2)^{2}|D|^{2} + \frac{|\nabla R + 2R\nabla f|^{2}}{2(n-1)}\right] \nonumber\\
 &&- \int_{\Omega\cap\{R<1\}} \Bigg[\frac{3}{2}|\nabla R|^{2} - 2(1-R)Ric(\nabla f,\,\nabla f)\nonumber\\
 &&+2div((1-R)^{2}\nabla f)\Bigg],\nonumber
\end{eqnarray}
and so
\begin{eqnarray}\label{ric>0}
 \int_{\partial \Omega}\langle(1-R)\nabla R,\,\nu\rangle  
 &=&-\int_{\Omega\cap\{R<1\}}\left[(n-2)^{2}|D|^{2} + \frac{|\nabla R + 2R\nabla f|^{2}}{2(n-1)}\right] \nonumber\\
 &&- \int_{\Omega\cap\{R<1\}} [\frac{3}{2}|\nabla R|^{2} - 2(1-R)Ric(\nabla f,\,\nabla f)].
\end{eqnarray}

Moreover, a straightforward computation yields to
\begin{eqnarray*}
 \int_{\partial \Omega}\langle(1-R)\nabla R,\,\nu\rangle  &\leq&-\int_{\Omega\cap\{R<1\}}\left[(n-2)^{2}|D|^{2} + \frac{|\nabla R + R\nabla f|^{2}}{2(n-1)}\right] \nonumber\\
 && - \frac{3}{2}\int_{\Omega\cap\{R<1\}}\left[|\nabla R|^{2}+ \frac{1}{(n-1)}R^2|\nabla f|^{2}\right]  \nonumber\\
 &&- \frac{1}{(n-1)}\int_{\Omega\cap\{R<1\}}R\langle\nabla R,\,\nabla f\rangle 
 + \int_{\Omega\cap\{R<1\}} 2(1-R)Ric(\nabla f,\,\nabla f)\nonumber\\
 &\leq&-\int_{\Omega\cap\{R<1\}}\left[(n-2)^{2}|D|^{2} + \frac{n|\nabla R + R\nabla f|^{2}}{2(n-1)}\right] \nonumber\\
 && - \int_{\Omega\cap\{R<1\}}|\nabla R|^{2} + \frac{(n-4)}{2(n-1)}\int_{\Omega\cap\{R<1\}}R^2|\nabla f|^{2}  \nonumber\\
 &&+ \frac{n-2}{(n-1)}\int_{\Omega\cap\{R<1\}}R\langle\nabla R,\,\nabla f\rangle \nonumber\\
 &&+ \int_{\Omega\cap\{R<1\}} 2(1-R)Ric(\nabla f,\,\nabla f).\nonumber\\
\end{eqnarray*}
On the other hand,
\begin{eqnarray*}
 \int_{\partial \Omega}(1-R)\langle\nabla R+R\nabla f,\,\nu\rangle  
 &\leq&-\int_{\Omega\cap\{R<1\}}\left[(n-2)^{2}|D|^{2} + \frac{n|\nabla R + R\nabla f|^{2}}{2(n-1)}\right] \nonumber\\
 && - \int_{\Omega\cap\{R<1\}}|\nabla R|^{2} + \frac{(n-4)}{2(n-1)}\int_{\Omega\cap\{R<1\}}R^2|\nabla f|^{2}  \nonumber\\
 &&+ \frac{n-2}{(n-1)}\int_{\Omega\cap\{R<1\}}R\langle\nabla R,\,\nabla f\rangle \nonumber\\
 &&- \int_{\Omega\cap\{R<1\}} (1-R)\langle\nabla R,\,\nabla f\rangle + \int_{\partial \Omega}(1-R)R\langle \nabla f,\,\nu\rangle.
\end{eqnarray*}

We also have the following identity:
\begin{eqnarray*}
\int_{\partial \Omega}(1-R)R\langle \nabla f,\,\nu\rangle 
&=&
\int_{\Omega\cap\{R<1\}}div((1-R)R\nabla f)\nonumber\\
&=& \int_{\Omega\cap\{R<1\}}[(1-R)R\Delta f + (1-R)\langle\nabla R,\,\nabla f\rangle- R\langle\nabla R,\,\nabla f\rangle]\nonumber\\
&=& \int_{\Omega\cap\{R<1\}}[R^{2}|\nabla f|^{2} + (1-R)\langle\nabla R,\,\nabla f\rangle- R\langle\nabla R,\,\nabla f\rangle].
\end{eqnarray*}
Therefore, 
\begin{eqnarray}\label{kahler}
 \int_{\partial \Omega}(1-R)\langle\nabla R+R\nabla f,\,\nu\rangle  
 &\leq&-\int_{\Omega\cap\{R<1\}}\left[(n-2)^{2}|D|^{2} + \frac{n|\nabla R + R\nabla f|^{2}}{2(n-1)}\right] \nonumber\\
 && - \int_{\Omega\cap\{R<1\}}|\nabla R|^{2} + \frac{3(n-2)}{2(n-1)}\int_{\Omega\cap\{R<1\}}R^2|\nabla f|^{2}  \nonumber\\
 &&+ \frac{1}{(n-1)}\int_{\Omega\cap\{R<1\}}2RRic(\nabla f,\,\nabla f)\\
  &=&-\int_{\Omega\cap\{R<1\}}\left[(n-2)^{2}|D|^{2}\right] \nonumber\\
 && + \int_{\Omega\cap\{R<1\}}\left[-\frac{(3n-2)}{2(n-1)}|\nabla R|^{2} + \frac{(n-3)}{(n-1)}R^2|\nabla f|^{2} \right] \nonumber\\
 &&+ \frac{(n+1)}{(n-1)}\int_{\Omega\cap\{R<1\}}2RRic(\nabla f,\,\nabla f).\nonumber
\end{eqnarray}

Considering \eqref{relacaof e R} we can infer that
\begin{eqnarray*}
 2Ric(\nabla f,\,\nabla f)\leq |\nabla R||\nabla f|.
\end{eqnarray*}
Thus,
\begin{eqnarray*}
 \int_{\partial \Omega}(1-R)\langle\nabla R+R\nabla f,\,\nu\rangle  
 &\leq&-\int_{\Omega\cap\{R<1\}}\left[(n-2)^{2}|D|^{2}\right] \nonumber\\
 && + \int_{\Omega\cap\{R<1\}}\Bigg[-\frac{(3n-2)}{2(n-1)}|\nabla R|^{2} +\frac{(n+1)R|\nabla f|}{(n-1)}|\nabla R|\nonumber\\
 &&+ \frac{(n-3)}{(n-1)}R^2|\nabla f|^{2}\Bigg].
\end{eqnarray*}

Now, from hypothesis $$\dfrac{(n+1)+\sqrt{(n-1)(7n-13)}}{3n-2}R|\nabla f|\leq|\nabla R|.$$
Therefore,
$$-\frac{(3n-2)}{2(n-1)}|\nabla R|^{2} +\frac{(n+1)R|\nabla f|}{(n-1)}|\nabla R|+ \frac{(n-3)}{(n-1)}R^2|\nabla f|^{2}\leq0.$$

 Consequently,
\begin{eqnarray*}
\int_{\Omega_\ell\cap\{R<1\}}\left[(n-2)^{2}|D|^{2}\right]\leq - \int_{\partial \Omega_\ell}(1-R)\langle\nabla R+R\nabla f,\,\nu\rangle.
\end{eqnarray*}
Now, consider an exhaustion
of $M$ by bounded domains $\Omega_\ell$  such that
	$$\displaystyle\lim_{\ell\rightarrow\infty} \int_{\partial \Omega_\ell}|\nabla R+R\nabla f|=0.$$
Then, making $\ell\rightarrow\infty$, we have
\begin{eqnarray*}
 &&\int_{\{R<1\}}\left[(n-2)^{2}|D|^{2}\right] \leq  - \displaystyle\lim_{\ell\rightarrow\infty} \int_{\partial \Omega_\ell}(1-R)\langle\nabla R+R\nabla f,\,\nu\rangle \nonumber\\
 &&\leq \left| \displaystyle\lim_{\ell\rightarrow\infty} \int_{\partial \Omega_\ell}(1-R)\langle\nabla R+R\nabla f,\,\nu\rangle\right| \\
 &&\leq \displaystyle\lim_{\ell\rightarrow\infty} \int_{\partial \Omega_\ell}(1-R)|\langle\nabla R+R\nabla f,\,\nu\rangle|\leq 
 \displaystyle\lim_{\ell\rightarrow\infty} \int_{\partial \Omega_\ell}|\nabla R + R\nabla f| =0.
\end{eqnarray*}
Since $M$ is complete, from \eqref{Rf} we can infer that $\{R<1\}$ is dense. Otherwise $f$ will be a constant function. Therefore, the $D$-tensor is identically zero.

To finish the proof, we need to invoke \cite[Theorem 1.1]{cao2020}, \cite[Theorem 1.4]{cao2013} and \cite[Theorem 1.2]{cao2011}. 
\end{proof}

\begin{proof}[{\bf Proof of Theorem \ref{thmMunteanu}}]
From \eqref{imptinf} we get
\begin{eqnarray*}
 &&\int_{\partial \Omega}(1-R)\langle\nabla R-2(1-R)\nabla f,\,\nu\rangle =-\int_{\Omega\cap\{R<1\}}\left[(n-2)^{2}|D|^{2} + \frac{|\nabla R + 2R\nabla f|^{2}}{2(n-1)}\right] \nonumber\\
 &-& \int_{\Omega\cap\{R<1\}} [\frac{3}{2}|\nabla R|^{2}+ 6(1-R)Ric(\nabla f,\,\nabla f)+2R(1-R)^{2}].
\end{eqnarray*}

Then, a straightforward computation yields to 
\begin{eqnarray*}
 &&\int_{\partial \Omega}(1-R)\langle\nabla R+2R\nabla f,\,\nu\rangle =-\int_{\Omega\cap\{R<1\}}\left[(n-2)^{2}|D|^{2} + \frac{|\nabla R + 2R\nabla f|^{2}}{2(n-1)}\right] \nonumber\\
 &-& \int_{\Omega\cap\{R<1\}} [\frac{3}{2}|\nabla R|^{2}+ 6(1-R)Ric(\nabla f,\,\nabla f)+2R(1-R)^{2}] + 2\int_{\partial \Omega}(1-R)\langle\nabla f,\,\nu\rangle.
\end{eqnarray*}
Then, we use that $\div [2(1-R)\nabla f]=-2\langle\nabla R,\,\nabla f\rangle + 2(1-R)R$ to obtain
\begin{eqnarray*}
 &&\int_{\partial \Omega}(1-R)\langle\nabla R+2R\nabla f,\,\nu\rangle =-\int_{\Omega\cap\{R<1\}}\left[(n-2)^{2}|D|^{2} + \frac{|\nabla R + 2R\nabla f|^{2}}{2(n-1)}\right] \nonumber\\
 &-& \int_{\Omega\cap\{R<1\}} [\frac{3}{2}|\nabla R|^{2}+ 2(1-3R)Ric(\nabla f,\,\nabla f)-2R^2(1-R)].
\end{eqnarray*}

Now, since $Ric>0$ and $-R\geq-1$ we can infer that
\begin{eqnarray*}
 &&-\int_{\partial \Omega}(1-R)\langle\nabla R+2R\nabla f,\,\nu\rangle \geq \int_{\Omega\cap\{R<1\}}\left[(n-2)^{2}|D|^{2} + \frac{|\nabla R + 2R\nabla f|^{2}}{2(n-1)}\right] \nonumber\\
 &+& \int_{\Omega\cap\{R<1\}} [\frac{3}{2}|\nabla R|^{2}-4Ric(\nabla f,\,\nabla f)-2R^2(1-R)].
\end{eqnarray*}

From Kato’s inequality, for the steady solitons we have
$$|Ric|^2=|\nabla^2f|^2\geq|\nabla|\nabla f||^2=|\nabla\sqrt{1-R}|^2=\frac{|\nabla R|^2}{4|\nabla f|^2}.$$
Hence,
\begin{eqnarray}\label{kato}
2Ric(\nabla f,\,\nabla f) = -\langle\nabla R,\,\nabla f\rangle\leq |\nabla R||\nabla f|\leq 2|Ric||\nabla f|^2.
\end{eqnarray}
So,
\begin{eqnarray*}
 &&-\int_{\partial \Omega}(1-R)\langle\nabla R+2R\nabla f,\,\nu\rangle \geq \int_{\Omega\cap\{R<1\}}\left[(n-2)^{2}|D|^{2} + \frac{|\nabla R + 2R\nabla f|^{2}}{2(n-1)}\right] \nonumber\\
 &+& \int_{\Omega\cap\{R<1\}} [\frac{3}{2}|\nabla R|^{2}-4|Ric||\nabla f|^2-2R^2|\nabla f|^2].
\end{eqnarray*}

Now, consider the bounded domains as geodesic balls, i.e., $|\partial\Omega_
    \ell|=w_{n-1}r(\ell)^{n-1}$, where $w_{n-1}$ stands for the volume of the $(n-1)$-sphere. Here, $\ell\to\infty$ implies that $r\to\infty.$ For a nontrivial gradient steady soliton with $Ric \geq 0$ and $\lim_{r\to\infty} R = 0$, we know that  $|Ric|^2\leq R^2$ (see \cite{carrillo2009} and \cite[page 12]{chan2019}). Thus, from \eqref{munteanu22} and \eqref{kato} we have
\begin{eqnarray*}
 \int_{\{R<1\}}\left[(n-2)^{2}|D|^{2}\right]&\leq&\lim_{\ell\to\infty}\int_{\partial \Omega_\ell}(|\nabla R|+2R|\nabla f|)\leq\lim_{\ell\to\infty}\int_{\partial \Omega_{\ell}}(\sqrt{2}R+2R)|\nabla f|\nonumber\\ &\leq&A(n)\displaystyle\lim_{r\rightarrow\infty}(2+\sqrt{2})w_{n-1}r^{n-1}(1+r)^{3(n+1)}e^{-r}\nonumber\\
 &\leq&c\displaystyle\lim_{r\rightarrow\infty}(1+r)^{2(2n+1)}e^{-r}=0,
\end{eqnarray*}
where we use that $|\nabla R|\leq\sqrt{2}R.$ Here, $c=A(n)(2+\sqrt{2})w_{n-1}$ and $A(n)$ a constant depending on $n.$

Then, $D=0$ and the proof follows like in Theorem \ref{thmA}.
\end{proof}

\begin{proof}[{\bf Proof of Corollary \ref{maintheoremkahler}}]
It is well-known that for K\"ahler manifolds with nonnegative Ricci curvature the following inequality holds:
$$2Ric(\nabla f,\,\nabla f)\leq R|\nabla f|^{2},$$
see Theorem 3.2 in \cite{Deruelle2012}. So, from \eqref{kahler} we get 
\begin{eqnarray*}
 \int_{\partial \Omega}(1-R)\langle\nabla R+R\nabla f,\,\nu\rangle  
 &\leq&-\int_{\Omega\cap\{R<1\}}\left[(n-2)^{2}|D|^{2} + \frac{n|\nabla R + R\nabla f|^{2}}{2(n-1)}\right] \nonumber\\
 && - \int_{\Omega\cap\{R<1\}}|\nabla R|^{2} + \frac{3n-4}{2(n-1)}\int_{\Omega\cap\{R<1\}}R^2|\nabla f|^{2}.
 \end{eqnarray*}
Assuming $\frac{1}{2}\sqrt{\frac{3n-4}{2(n-1)}}Rg\leq Ric$, from \eqref{relacaof e R} we have
 \begin{eqnarray}\label{pinchedricci}
 \sqrt{\frac{3n-4}{2(n-1)}}R \leq 2Ric(\frac{\nabla f}{|\nabla f|},\,\frac{\nabla f}{|\nabla f|})= \frac{-1}{|\nabla f|^{2}}\langle\nabla R,\,\nabla f\rangle\leq \frac{|\nabla R|}{|\nabla f|}.
 \end{eqnarray}
 
 It is worth to say that for any steady gradient Ricci soliton (not necessarily K\"ahler), the above inequality holds considering $\frac{1}{2}\sigma Rg\leq Ric$, where $\sigma$ is given by Theorem \ref{thmA}. Thus, $\sigma R|\nabla f|\leq|\nabla R|.$

Now, the result follows by the same steps of Theorem {\ref{thmA}}. In fact, considering the bounded domains as geodesic balls, i.e., $|\partial\Omega_
    \ell|=w_{n-1}r(\ell)^{n-1}$, where $w_{n-1}$ stands for the volume of the $(n-1)$-sphere, from \eqref{eqR1} we can conclude that \eqref{asympBene} satisfies
    \begin{equation}\label{expdecay}
    \displaystyle\lim_{\ell\rightarrow+\infty} \int_{\partial \Omega_\ell}(|\nabla R| + R|\nabla f|) = \displaystyle\lim_{r\rightarrow+\infty}w_{n-1}r^{n-1}\left[o(e^{-r})+o(e^{-r})\sqrt{(1-o(e^{-r})})\right] = 0,\nonumber
    \end{equation}
    where we assume that $r\to\infty$ when $\ell\to\infty$.
Here, we consider that the scalar curvature has an exponential decay at infinity, i.e., 
    \begin{eqnarray*}
    R = o(e^{-r}),\quad r\rightarrow\infty,
    \end{eqnarray*}
    where $r$ stands for the geodesic distance from a fixed point, see \cite[Theorem 9.56]{CLN1}.

Therefore, we obtain $\nabla R+R\nabla f =0$ and $D=0$. Moreover, 
$$\frac{3n-4}{2(n-1)}R^{2}|\nabla f|^{2} = |\nabla R|^{2}.$$
So, either $n=2$ or $(M,\,g)$ is Ricci flat and $f$ a constant function. 
\end{proof}

\begin{acknowledgement}
The authors wants to thank Professor E. Ribeiro Jr. for
the helpful remarks and discussions. Jeferson Poveda was supported by PROPG-CAPES [Finance Code 001]. Benedito Leandro was partially supported by Brazilian National Council for
Scientific and Technological Development (CNPq Grant 403349/2021-4).
\end{acknowledgement}

\end{document}